\documentclass{amsart}

\usepackage{amsmath, amsfonts, amssymb, amsthm, mathtools}
\usepackage{enumerate}  

\usepackage{multibib}
\newcites{recent}{Additional References}

\usepackage[shortlabels]{enumitem}
\usepackage[colorlinks]{hyperref}
\usepackage[capitalise]{cleveref}

\usepackage{tikz-cd}
\usetikzlibrary{decorations.pathmorphing}

\newtheorem{theorem}[subsection]{Theorem}
\newtheorem{proposition}[subsection]{Proposition}
\newtheorem{lemma}[subsection]{Lemma}
\newtheorem{corollary}[subsection]{Corollary}

\theoremstyle{definition}

\newtheorem*{conjecture}{Conjecture}
\newtheorem{question}[subsection]{Question}
\newtheorem{definition}[subsection]{Definition}
\newtheorem*{ack}{Acknowledgements}

\theoremstyle{remark}

\numberwithin{equation}{subsection}

\usepackage{mathtools}
\usepackage{todonotes}

\newcommand{\aq}{\operatorname{T}}
\newcommand{\coker}{\operatorname{coker}}
\newcommand{\cotam}[2]{\operatorname{C}_{#1}(#2)}
\newcommand{\depth}{\operatorname{depth}}
\newcommand{\ext}{\operatorname{Ext}}
\newcommand{\fm}{{\mathfrak m}}
\newcommand{\fp}{{\mathfrak p}}
\newcommand{\grade}{\operatorname{grade}}
\newcommand{\height}{\operatorname{height}}
\newcommand{\hh}{\operatorname{H}}
\newcommand{\Hom}{\operatorname{Hom}}
\newcommand{\idmap}{\operatorname{id}}
\newcommand{\im}{\operatorname{im}}
\newcommand{\kos}{\operatorname{K}}
\newcommand{\pdim}{\operatorname{pd}}
\newcommand{\poin}{\operatorname{P}}
\newcommand{\rank}{\operatorname{rank}}
\newcommand{\spec}{\operatorname{Spec}}
\newcommand{\tor}{\operatorname{Tor}}

\begin{document}

\title[Module of differentials]{Homological properties of the \\ module of differentials}
\author[Herzog]{J\"urgen Herzog\\
Appendix by Benjamin Briggs and Srikanth B. Iyengar}

\date{\today}


\begin{abstract} 
These notes were produced by J\"urgen Herzog to accompany his lectures in Recife, Brazil, in 1980, on the homological algebra of noetherian local rings.  They are are concerned with two conjectures made by Wolmer Vasconcelos: if the conormal module of a local ring has finite projective dimension, or if the module of differentials, taken over an appropriate field, has finite projective dimension, then the ring must be complete intersection. The notes present an accessible and self-contained account of the strongest results known at the time in connection with these problems; this includes a number of ideas that have not appeared elsewhere. In the last section, Herzog turns his attention to the cotangent complex, and conjectures himself that if the cotangent complex of a local ring has bounded homology groups, then the ring must be complete intersection. Among other results, he proves that the conjecture holds for local rings of characteristic zero over which all modules have rational Poincar\'e series.

Sadly J\"urgen Herzog passed away in April of 2024. The notes in this form have been prepared in his memory, newly typeset and lightly edited. A short appendix has been added to survey some of the results of the intervening decades.
\end{abstract}

\maketitle

\setcounter{tocdepth}{1}
\tableofcontents

\section*{Preface}
These are a lightly edited version of notes written by J\"urgen Herzog on homological properties of the conormal module, the module of differentials, and the cotangent complex of a noetherian local algebra. They were originally delivered as lectures at the 6th Algebra School of the Brazilian Mathematical Society, in Recife, Brazil, in July of 1980, and appeared in \citerecent{sociedade1981atas}. The notes have been rather influential in the decades since, and quoted in many subsequent works, although they were never easily accessible.

Some of the results have been subsumed in later work, but not all: the theorem of Platte presented in \cref{se:platte}---that an algebra must be quasi-Gorenstein if it has a module of differentials with finite projective dimension---remains one of the strongest known; and likewise the results of \cref{se:cotangent} on the homology of the cotangent complex, joint with Steurich, have not been substantially improved upon. Even aside from the state-of-the-art, the presentation of the basic theory is as useful as ever, and the ideas retain their freshness, so it was thought worthwhile to publish a version of these lectures.

Herzog himself had intended to update the notes for publication in this volume of proceedings. Regrettably he passed away suddenly in April of 2024,
before he could undertake this task. At this point we (Briggs and Iyengar) decided to prepare the notes to be published, albeit without editing them to the extent that Herzog himself may have done.

Sections \ref{se:introduction} to \ref{se:cotangent}, and the first part of the references,  are essentially a faithful reproduction of Herzog's notes. We have taken the liberty to update some notation, correct minor typographical errors, and add clarifying sentences in a few places in the text; when appropriate, these appear in footnotes.

In ~\cref{se:recent} we give a short update on the three conjectures discussed in the body of the text; two due to Vasconcelos on the conormal module and the module of differentials, and one due to Herzog on the cotangent complex, that first appears at the very end of \cref{se:cotangent}. The references that appear here are listed separately.

\section{Introduction}
\label{se:introduction}
In these lectures we present some results concerned with a conjecture of W.~V.~Vasconcelos and some related topics. In \cref{se:cotangent} we consider the cotangent functors $\aq_i$. This  part is joint work with M.~Steurich. Our assumptions unless otherwise stated will be the following: 
 \begin{itemize}
     \item $R$ is a  noetherian local ring;
     \item $I$ is a proper ideal of finite projective dimension;
     \item $S\coloneqq R/I$.
\end{itemize}
All modules will be finitely generated. Given an $R$-module $M$ we use the notation:
\begin{itemize}
    \item $M^*\coloneqq \Hom_R(M,R)$ is the dual of $M$;
    \item $\tau_R(M)$ is the torsion submodule
    \[
    \{m\in M \mid rm=0 \text{ for a nonzero divisor $r$ of $R$}\};
    \]
    \item $\ell_R(M)$ is the length of $M$;
    \item $\mu_R(M)$ is the minimum number of generators of $M$;
    \item $\pdim_R(M)$ is the projective dimension of $M$.
\end{itemize}

The main question we are dealing with this

\begin{conjecture}[C1]
\label{con:C1}
   If $\pdim_S(I/I^2)<\infty$, then  $I$ is a complete intersection, that is to say, $I$ can be generated by a regular sequence.
\end{conjecture}

   Observe that if $I$ is a complete intersection, then $I/I^2$ is even free as an $S$-module. This statement has a converse:

   \begin{theorem}[D.~Ferrand~\cite{Ferrand:1967}, W.~V.~Vasconcelos~\cite{Vasconcelos:1967}]
   \label{th:ferrand-vasconcelos}
   If $I/I^2$ is free as an $S$-module, then $I$ is a complete intersection.
   \end{theorem}

The following result, which strengthens Theorem \ref{th:ferrand-vasconcelos}, is due to T.~H.~Gulliksen~\cite[Proposition 1.4.9]{Gulliksen:1971}.

\begin{theorem}\label{1.2}
    If $\pdim_S(I/I^2)\le 1$, then $I$ is a complete intersection.
\end{theorem}

\begin{proof}[Sketch of a proof]
Suppose $I$ is minimally generated by $x_1,\dots,x_n$. Let $\hh_1$ be the first homology of the Koszul complex $\kos(x_1,\dots,x_n; R)$. Then there is a natural exact sequence of $S$-modules
\[
H_1\xrightarrow{\ \iota\ } S^n \xrightarrow{\ \varepsilon\ } I/I^2 \longrightarrow 0
\]
where $\varepsilon$ is the obvious epimorphism. To define $\iota$ consider a $1$-cycle $z=\sum_{i=1}^nr_iT_i$. We define $[z]=(\overline{r}_1,\dots,\overline{r}_n)$. Here $[z]$ is the homology class of $z$ and $\overline{r}_i$ the residue class of $r_i$ modulo $I$. Since  $\pdim_S(I/I^2)\le 1$ either $H_1=0$ or it contains a free summand $F\ne 0$. Given Theorem~\ref{th:ferrand-vasconcelos}, we can suppose $H_1\ne 0$. Then, from the proof of ~\cite[Proposition 1.4.9]{Gulliksen:1971}  it follows  that $I$ is a complete intersection.
\end{proof}

Conjecture (C1) is not settled for ideals with $\pdim_S (I/I^2)=2$\footnote{This has changed; see Section~\ref{se:recent}}.

Recall that $I$ is called \emph{perfect} if $\pdim_R(R/I)=\height(I)$, the height of $I$. Complete intersections are perfect so one might try to prove at least that $I$ is perfect if $\pdim_S (I/I^2)$ is finite. For a certain class of ideals W.~V.~Vasconcelos \cite{Vasconcelos:1978a,Vasconcelos:1978b} could prove this.

\begin{theorem}\label{1.3}
\ 
\begin{enumerate}[\quad\rm(a)]
    \item 
        If $\pdim_R (I)\le 1$ and $\pdim_S (I/I^2)<\infty$, then $I$ is a complete intersection. 
    \item
        If $\pdim_R(I) = 2$ and $\pdim_S(I/I^2)<\infty$ then $I$ is perfect. If moreover $I$ is a Gorenstein ideal---that is $\ext^3_R(S,R)\cong S$---and $2$ is
unit in $R$, then $I$ is a complete intersection.
\end{enumerate}
   
\end{theorem}
  
We will see later in \cref{se:platte} that if the ambient ring $R$ is Gorenstein, and $I$ is a perfect ideal with $\pdim_S(I/I^2)<\infty$, then
$I$ is necessarily a Gorenstein ideal, see \cref{3.3,3.4}. Also the requirement that $2$ is unit in $R$ can be skipped as we will see in \cref{se:linkage}. Thus we have: If $R$ is a Gorenstein ring, $\pdim_R(I)=2$ and $\pdim_S(I/I^2)<\infty$ then $I$ is a complete intersection. Therefore Conjecture (C1) is settled in the case that $R$ is Gorenstein and $\pdim_R(I)=2$.

There is another case where the conjecture can be answered in the affirmative. We call an ideal $I$ an \emph{almost complete intersection} if $\mu_R(I)\le\grade(I)+1$, where $\grade(I)$ is the maximal length of a regular sequence in $I$. The following result is essentially due to Y. Aoyama~\cite{Aoyama:1977} and T. Matsuoka~\cite{Matsuoka:1977}.

\begin{theorem}
If $I$ is an almost complete intersection and $\pdim_S (I/I^2)<\infty$, then $I$ is a complete intersection.    
\end{theorem}

We now discuss briefly the related conjecture for the module of differentials $\Omega_{S/k}$. Here we restrict our attention to the case that $S$ is the local ring of an affine $k$-algebra, $k$ being a field of characteristic $0$. Thus $S = R/I$, where
\[
R = k[x_1,\dots,x_n]_\fp \quad\text{with}\quad \fp \in \spec k[x_1,\dots,x_n]\,.
\]

\begin{conjecture}[C2]
\label{C2}
     If $\pdim_S \Omega_{S/k}<\infty$, then $I$ is a complete intersection.
\end{conjecture}

Observe that the assumption implies already that $S$ is equidimensional and reduced and that $\rank \Omega_{S/k}=\dim S$.

There is a natural exact sequence relating $I/I^2$ and $\Omega_{S/k}$:
\begin{align*}
    &I/I^2\xrightarrow{\ \delta\ } \Omega_{R/k}\otimes_RS \longrightarrow \Omega_{S/k}\longrightarrow 0 \\
    &f+I^2\mapsto d(f)\otimes 1
\end{align*}

Counting ranks one finds that $\ker(\delta) = \tau(I/I^2)$. If it  happens that $I/I^2$ is torsionfree (see for instance \cref{2.4}), then $\delta$
is injective and (C1) is equivalent to (C2).

Although it is not quite clear how (C1) is related to (C2) in general, (C2) can be proved also in all the above mentioned cases. In the following we focus on (C1), since this conjecture can be formulated in more general situations.

\section{Linkage of ideals}
\label{se:linkage}

In this section we suppose that $R$ is a local Gorenstein ring. We show that with the technique of linkage, introduced by C.~Peskine and L.~Szpiro~\cite{Peskine/Szpiro:1974} one can produce a large class of examples of ideals for which (C1) is true.

\begin{definition}
    \label{2.1}
    Let $I$ be an ideal of grade $g$. The ideal $J$ is \emph{linked} to $I$ if $I$ contains a regular sequence $\boldsymbol{x}\coloneqq x_1,\dots,x_g$ with $I\supseteq (\boldsymbol{x})$ and $I_\fp=(\boldsymbol{x})_\fp$ for all minimal primes of $I$, and such that $J=((\boldsymbol{x}):I)$.
\end{definition}

We need the following result which is similar to \cite[2.6]{Peskine/Szpiro:1974}.

\begin{theorem}
    \label{2.2}
    Suppose $J$ is linked to $I$. If $I$ is perfect of grade $g$ then also $J$ is perfect of grade $g$ and $I=((\boldsymbol{x}):J)$.
\end{theorem}

For the convenience of the reader we give here a short proof.

\begin{proof}
We have $\grade J\ge g$, since $(\boldsymbol{x})\subseteq J$. Suppose $\grade J\ge g+1$. Then there exists an element $y\in J$ which is regular modulo $(\boldsymbol{x})$. Since $yI\subseteq (\boldsymbol{x})$ this would imply that $I=(\boldsymbol{x})$, a contradiction. We put $S=R/I$ and $T=R/J$. To prove that $J$ is perfect we need only to show that $\pdim_R (T)\leq g$. Let
\[
F\coloneqq 0\longrightarrow F_g\longrightarrow \cdots \longrightarrow F_1\longrightarrow F_0\longrightarrow 0
\]
be a minimal free $R$-resolution of $S$ and let $K=K(\boldsymbol{x};R)$ be the Koszul complex on the sequence $\boldsymbol{x}$. We have a comparison map $\varphi\colon K\to F$ with $\varphi_0=\idmap_R$. Let $C$ be the mapping cone of $\varphi$. The dual complex
\[
C^* =  0\longrightarrow F_0^*\longrightarrow F_1^*\oplus K_0^*\longrightarrow \cdots 
    \longrightarrow F_g^*\oplus K^*_{g-1}\longrightarrow K_g^* \longrightarrow 0
\]
is a free $R$-resolution of $T$. In fact, we have an exact sequence of complexes
\[
0\longrightarrow K^*[-1]\longrightarrow C^*\longrightarrow F^*\longrightarrow 0
\]
where $ K^*[-1]$ means $K^*$ shifted to the right by $1$. From this it is immediate that $C^*$ is acyclic, except in degree $-g-1$. We also obtain the exact sequence
\[
0\longrightarrow \hh_{-g}(F^*)\longrightarrow \hh_{-g}(K^*)\longrightarrow \hh_{-g-1}(C^*)\longrightarrow 0\,.
\]
Since $\hh_{-g}(K^*)\cong R/(\boldsymbol{x})$ and 
\[
\hh_{-g}(F^*)\cong \ext^g_R(S,R)\cong \Hom_R(S,R/(\boldsymbol{x}))\cong J/(\boldsymbol{x})\,,
\]
it follows that $\hh_{-g-1}(C^*)\cong R/J\cong T$. The left-most differential in $C^*$ is split injective, hence can be canceled. Thus $\pdim_R(T)\le g$.

The equality $I=((\boldsymbol{x}):J)$ follows easily from local duality: Apply \cite[6.8]{Herzog/Kunz:1971} to $R/(\boldsymbol{x})$ and the module $J/(\boldsymbol{x})$.
\end{proof}

It is clear that $J$ coincides with $(\boldsymbol{x})$ generically if $I$ does. Thus linkage is a symmetric relation. It is however neither reflexive nor transitive. Let $\sim$ denote the smallest equivalence relation that contains the linkage relation. We say $J$ \emph{belongs to the linkage class of} $I$ if $J\sim I$.

\begin{theorem}
    \label{2.3}
    If $I$ belongs to the linkage class of a complete intersection and $\pdim_S (I/I^2)<\infty$, then $I$ is a complete intersection.
\end{theorem}

Let $K_P$ denote the canonical module of a local ring $P$, see \cite{Herzog/Kunz:1971}. The above theorem is a consequence of the following two results.

\begin{theorem}[R.~O.~Buchweitz~\cite{Buchweitz:1981}]
    \label{2.4}
    Suppose $I$ and $J$ are perfect ideals belonging to the same linkage class. Let $S=R/I$ and $T=R/J$. Then $I\otimes_RK_S$ is a Cohen-Macaulay module if and only if $J\otimes_RK_T$ is a Cohen-Macaulay module.
\end{theorem}

\begin{theorem}[\protect{\cite[4.1]{Herzog:1980}}]
    \label{2.5}
    If $I$ is a perfect ideal such that $I\otimes_RK_S$ is a Cohen-Macaulay module and such that $\pdim_S (I/I^2)<\infty$, then $I$ is a complete intersection.
\end{theorem}

Under the same hypotheses as \cref{2.3}, it follows from \cref{3.4,3.5} of the next section that $I$ is a Gorenstein ideal, and hence that $I\otimes_R K_S\cong I/I^2$. Knowing this,  \cref{2.3} also follows directly from \cref{2.4}. 

\begin{proof}[Proof of \cref{2.4}]
We will use the following Cohen-Macaulay criterion \cite[1.1]{Herzog:1978}: Let $S$ be a Cohen-Macaulay ring, $M$ an $S$-module with rank and $\boldsymbol{s}$ a system of parameters of $S$. Then $M$ is Cohen-Macaulay if and only if 
\[
\ell (M/(\boldsymbol{s})M) = \rank(M) \cdot \ell (S/(\boldsymbol{s})S).
\]
The $S$-module $I\otimes_RK_S\cong I/I^2\otimes_S K_S$ has rank $g=\grade I$, since we assume that $I$ is generically a complete intersection.

Applying the criterion above we need only to show the following: When $I\subset R$ is primary to the maximal ideal of $R$ and $J=((\boldsymbol{x}):I)$ is linked to $I$, then
\[
\ell(I\otimes_R K_S) -g\cdot \ell(S) = \ell(J\otimes_R K_T) -g\cdot \ell(T),
\]
where $g=\grade I=\grade J =$ dimension of $R$.

To check this, we put $\overline{R}=R/(\boldsymbol{x})$ and obtain an exact sequence
    \[
    0\longrightarrow K_T \longrightarrow \overline{R}\longrightarrow S\longrightarrow 0\,.
    \]
In fact, $K_T\cong \Hom_R(T,\overline{R})\cong I/(\boldsymbol{x})$. We tensor this sequence with $T$ and obtain a long exact sequence
    \[
    \tor_2^R(\overline{R},T)\xrightarrow{\alpha_2} \tor_2^R(S,T)\to \tor^R_1(K_T,T) \to
        \tor_1^R(\overline{R},T)\xrightarrow{\alpha_1} \tor_1^R(S,T)\to 
    \]
By symmetry there is a similar exact sequence
    \[
    \tor_2^R(\overline{R},S)\xrightarrow{\beta_2} \tor_2^R(T,S)\to \tor^R_1(K_S,S)\to
        \tor_1^R(\overline{R},S)\xrightarrow{\beta_1} \tor_1^R(T,S)\to 
    \]
Moreover one has isomorphisms 
    \begin{gather*}
    I\otimes_RK_S\cong \tor^R_1(K_S,S)\quad\text{and}\quad  J\otimes_R K_T \cong \tor^R_1(K_T,T)\\
     S^g \cong \tor^R_1(\overline{R},S)  \quad\text{and}\quad T^g \cong \tor^R_1(\overline{R},T)\,.
    \end{gather*}
Thus the assertion follows once we have shown that $\im(\alpha_i)=\im(\beta_i)$.

To see this consider the following diagram, that commutes up to a sign, where all the homomorphisms are just the natural ones:
    \[
    \begin{tikzcd}[row sep=tiny]
        &\tor^R_i(\overline{R},T) \arrow[r,"\alpha_i"] & \tor^R_i(S,T)\arrow[dd, leftrightarrow,"\cong"] \\
    \tor^R_i(\overline{R},\overline{R})\arrow[ur,"\sigma_i"] \arrow[dr,"\tau_i" swap]& \\
        &\tor^R_i(\overline{R},S) \arrow[r,"\beta_i" swap] & \tor^R_i(T,S)
    \end{tikzcd}
    \]
Using the Koszul complex $K(\boldsymbol{x},R)$, which resolves $\overline{R}$, to compute $\tor_i^R(\overline{R},\overline{R})$, $\tor_i^R(\overline{R},S)$, and $\tor_i^R(\overline{R},T)$ one sees that $\sigma_i$ and $\tau_i$ are surjective. Hence the conclusion.
\end{proof}

\begin{corollary}
\label{2.6}
If $I$ belongs to the linkage class of a complete intersection and $\pdim_S\Omega_{R/k}<\infty$, then $I$ is a complete intersection.
\end{corollary}

\begin{proof}
From (2.3) and (3.17), in Section 3, we conclude that $I/I^2$ is a Cohen-Macaulay module. The assertion then follows from the discussion at the end of \cref{se:introduction}.            
\end{proof}

\section{A theorem of E.~Platte}
\label{se:platte}

In this section we present a recent result of E.~Platte~\cite{Platte:1980}. Besides our general assumptions from the introduction we will always assume that $R$ is a Gorenstein ring. As before, we assume that $I$ is an ideal in $R$ of finite projective dimension. The ring $S=R/I$ is called \emph{quasi-Gorenstein} if $\ext^g_R(S,R)\cong S$ where $g=\height I$. Of course we have
\[
\text{$S$ Gorenstein } \Longleftrightarrow \text{ $S$ quasi-Gorenstein and Cohen-Macaulay.}
\]

Here is Platte's result.

\begin{theorem}
\label{3.1}
With $R$ and $S$ as above, if the $S$-module $I/I^2$ has finite projective dimension, then the ring $S$ is quasi-Gorenstein.\footnote{This appears as \cite[Korollar 3]{Platte:1980}, in the generality of analytic algebras generically separable over a field.}
\end{theorem}

The essential part in the proof is the following result.

\begin{lemma}
\label{3.2}
Let $R$ be any noetherian ring and let
\[
0\longrightarrow N\longrightarrow F\longrightarrow M\longrightarrow 0
\]
be an exact sequence, where $M$ is a module of rank $r$ and $F$ is a free module of rank $n$. If $M$ is free for all $\fp\in\spec R$ with $\depth R_\fp\le 1$, then there is a canonical isomorphism
\[
(\wedge^r M)^* \cong (\wedge^{n-r}N)^{**}\,.
\]
\end{lemma}

\begin{proof}
We define a mapping
\[
\varphi\colon \wedge^rM\longrightarrow \Hom_R(\wedge^{n-r}N,\wedge^nF)\cong (\wedge^{n-r}N)^*\,.
\]
Given $m_1\wedge \cdots \wedge m_r\in \wedge^rM$, choose $f_i\in F$ with $f_i\mapsto m_i$, then
\[
\varphi(m_1\wedge \cdots\wedge m_r)(w_1\wedge\cdots\wedge w_{n-r})=
    f_1\wedge\cdots \wedge f_r\wedge w_1\wedge \cdots \wedge w_{n-r}
\]
for all $w_1\wedge\cdots\wedge w_{n-r}$ in $\wedge^{n-r}N$, identified with their images in $F$. It is an easy matter to see that $\varphi$ is well-defined and that $\varphi$ is an isomorphism if $M$ is free. Since $\varphi$ is compatible with localizations, our conditions imply that the dual of $\varphi$ is an isomorphism.
\end{proof}

\begin{corollary}
\label{3.3}
If $R$ is local and $M$ satisfies the conditions of \eqref{3.2} and $\pdim_RM<$ is finite, then $(\wedge^rM)^*\cong R$.
\end{corollary}

The proof follows by induction on projective dimension applying the lemma to $M$ and to the syzygies of $M$. 

With our standard assumptions on $R$ and $I$ we have

\begin{corollary}
\label{3.4}
If $\height I=g$ and $\pdim_S(I/I^2)$ is finite, then
\[
(\wedge^g I/I^2)^*\cong S\,.
\]
\end{corollary}

\begin{proof}
      The assumption $\pdim_S(I/I^2)<\infty$ guarantees that $\rank_S(I/I^2)=g$ by Theorem \ref{th:ferrand-vasconcelos}. If $\fp\in\spec S$ satisfies $\depth S_\fp\le 1$, then $\pdim_{S_\fp}(I/I^2)\le 1$ and hence $(I/I^2)_\fp$ is free by Theorem \ref{1.2}. We may therefore apply \cref{3.3}.
\end{proof}

    The proof of Theorem \ref{3.1} is now a consequence of the following well-known result. The ideal $I$ is said to be equidimensional if $\height \mathfrak{p} =g$ for all associated primes $\mathfrak{p}$ of $I$; in particular, $I$ has no embedded primes.

\begin{proposition}
    \label{3.5}
If $\height I=g$, then there is a natural homomorphism
\[
\psi\colon \ext^g_R(S,R)\longrightarrow (\wedge^g I/I^2)^*\,.
\]
Moreover $\psi$ is an isomorphism if $I$ is equidimensional and $(I/I^2)_\fp$ is free for all $\fp\in\spec S$ with $\depth S_\fp\le 1$.
\end{proposition}

\begin{proof}
    We will exhibit $\psi$ as a composition of three mappings. Since $\tor^R(S,S)$ is a skew-symmetric algebra with $\tor^R_1(S,S)\cong I/I^2$, there is a natural mapping
    \[
    \wedge^g I/I^2\longrightarrow \tor^R_g(S,S)\,.
    \]
    Dualizing, we obtain a  natural mapping
    \[
    \alpha\colon \tor^R_g(S,S)^*\longrightarrow (\wedge^g I/I^2)^*\,.
    \]
    If $I/I^2$ is free then $I$ is a complete intersection and $\tor^R_g(S,S)\cong \wedge^g I/I^2$. Therefore $\alpha$ is an isomorphism under the extra conditions of \eqref{3.5}.
    
    Next let $F$ be a free $R$-resolution of $S$. Then there is a natural map
    \[
    \sigma\colon \tor^R_g(S,S)=\hh_g(F\otimes_RS)\longrightarrow \Hom_S(\hh^g(F^*),S)=\ext^g_R(S,R)^*\,.
    \]
If $[z]$ is in  $\hh_g(F\otimes_RS)$, choose $\widetilde{z}\in F$ with $z=\widetilde{z}\otimes 1$ and define
\[
\sigma([z])([w])=w(\widetilde{z})\otimes 1\quad \text{for $[w]\in \hh^g(F^*)$.}
\]
Set 
\[
\beta=\sigma^*\colon \ext^g_R(S,R)^{**}\longrightarrow \tor^R_g(S,S)^*\,.
\]
Finally we define $\psi=\alpha\circ\beta\circ\gamma$, where
\[
\gamma\colon \ext^g_R(S,R)\longrightarrow \ext^g_R(S,R)^{**}
\]
is the natural homomorphism to the bidual module. Again under the extra assumptions of \eqref{3.5} the maps $\beta$ and $\gamma$ are isomorphisms (see \cite[Cor.~3.2]{Andrade/Simis/Vasconcelos:1981}) and hence $\psi$ is an isomorphism. 
\end{proof}

Let us assume now that $R$ and $S$ are in Conjecture~(C2); thus $S=R/I$ where 
\[
R = k[x_1,\dots, x_n]_\fp\quad\text{with}\quad \fp \in  \spec k[x_1,\dots, x_n]
\]
and $k$ is a field of characteristic zero. Corresponding to \eqref{3.1}  we have

\begin{theorem}
\label{3.6}
If $\pdim_S \Omega_{S/k}<\infty$, then $S$ is quasi-Gorenstein.\footnote{This appears as \cite[Satz 1]{Platte:1980}, in the generality of analytic algebras generically separable over a field.}
\end{theorem}

\begin{proof}
Let $U$ be the kernel of the canonical epimorphism
\[
\Omega_{R/k}\otimes_RS \longrightarrow \Omega_{S/k}\longrightarrow 0\,.
\]
Again with $g=\height I$ we know that $(\wedge^g U)^*\cong S$ by \eqref{3.3}. Let $T$ be the kernel of the canonical epimorphism $\wedge^g I/I^2\to \wedge^g U$, then we obtain the exact sequence
\[
0\longrightarrow (\wedge^g U)^*\longrightarrow (\wedge^g I/I^2)^*\longrightarrow T^*\,.
\]
Since $\pdim_S\Omega_{S/k}<\infty$ we conclude that $T_\fp=0$ for $\fp\in \spec S$ with $\depth S_\fp=0$. Hence 
\[
T^*=0 \quad\text{and hence}\quad (\wedge^g I/I^2)^*\cong S\,.
\]
The assertion follows from \eqref{3.5}.
\end{proof}

We want to mention at this place that the results proved so far imply that Conjecture (C1) holds for ideals in Gorenstein rings of dimension $\le 5$. 

\section{The cotangent functors \texorpdfstring{$\aq_i$}{Ti}}
\label{se:cotangent}

The study of the cotangent modules $\aq_i(S/R,S)$ leads to a complex whose $0$th homology is $\Omega_{S/R}$. This complex may be considered to be an approximation of a free $S$-resolution of $\Omega_{S/R}$. We begin by giving a short sketch of the definition of the functors $T_i$, following V.\ P.\ Palamodov \cite{Palamodov:1976}.

In this section we will always assume that $R$ contains the rational numbers. To a ring homomorphism $\varphi\colon R\to S$ and an $S$-module $M$ there is assigned a sequence of $S$-modules $\aq_i(S/R,M)$, for $i=0,1,2,\ldots$, with the following properties:

\begin{enumerate}[(i)]
    \item $\aq_0(S/R,M)=\Omega_{S/R}\otimes_SM$,
    \item To any short exact sequence of $S$-modules
    \[
    0 \longrightarrow  M_1 \longrightarrow M_2\longrightarrow M_3\longrightarrow  0
    \]
    belongs a long exact sequence of $S$-modules
    \[
    \longrightarrow \aq_{i+1}(S/R,M_3) \longrightarrow \aq_i(S/R,M_1)\longrightarrow  \aq_i(S/R,M_2)\longrightarrow \aq_i(S/R,M_3)\longrightarrow 
    \]
    \item To a sequence of ring homomorphisms $R\to S\to T$ and a $T$-module $M$ belongs a long exact sequence
    \[
     \longrightarrow \aq_{i+1}(T/S,M) \longrightarrow  \aq_i(S/R,M)\longrightarrow  \aq_i(T/R,M) \longrightarrow  \aq_i(T/S,M)\longrightarrow  
    \]
    This sequence is called the \emph{Zariski sequence}\footnote{Also known as the Jacobi-Zariski exact sequence.}.
\end{enumerate}
Of course this list of properties is by far not complete. Among other things, $\aq_i$ is a functor in all three variables.

\subsubsection*{The definition of $\aq_i$} Given $\varphi\colon R\to S$, one constructs a \emph{resolvent} of $\varphi$. This is a commutative diagram
\[
\begin{tikzcd}[row sep = 3mm]
    & X \ar[dr,"\varepsilon"] &\\ 
    R \ar[ur,"\iota"] \ar[rr,"\varphi"]&&S
\end{tikzcd}
\]
where $X$ is a free dg algebra over $R$ and $\varepsilon$ is a surjection. Moreover it is required that $\hh_i(X)=0$ for $i>0$ and that $\varepsilon$ induces an isomorphism $\hh_0(X)\to S$.

Remember that a \emph{dg algebra} $X$ is a graded skew-symmetric\footnote{That is, graded-commutative: $ab=ba^{(\deg a)(\deg b)}$ for all homogeneous $a,b\in X$.} algebra with a differential $\delta$ of degree $-1$ such that
\[
\delta(ab) = (\delta a)b+(-1)^{\deg a} a(\delta b)
\]
for $a,b\in X$ and $a$ homogeneous. Instead of explaining what a free dg algebra is, we describe how one obtains a resolvent of $\varphi$. The reader who wants to know more detail is referred to \cite{Gulliksen/Levin:1969}.

We first describe the adjunction of variables: Given a dg algebra $X$ and a cycle $z\in X_i$, the dg algebra $X'=X\langle T ~|~ \delta(T) = z\rangle$ is obtained by adjoining the variable $T$ to $X$ in order to kill the cycle $z$. If $i$ is even we put $\deg T = i+1$ and let 
\[
X'_j = X_j\oplus X_{j-i-1}T, \ \ T^2=0,\ \text{ and }\ \ \delta(T)=z.
\]
If $i$ is odd we put $\deg T = i+1$ and let 
\[
X'_j = X_j\oplus X_{j-(i+1)}T\oplus  X_{j-2(i+1)}T^2\oplus \cdots\ \ \text{ with }\ \ T^iT^j=T^{i+j}\text{ and }\ \ \delta(T^i)=izT^{i-1}.
\]
With this data $X'$ is a well-defined dg algebra containing $X$. The resolvent of $\varphi$ is obtained by adjoining variables:
\begin{enumerate}[Step 1:]
    \item Adjoin sufficiently many variables of degree $0$ and define a ring homomorphism
    \[
    \phi\colon R[\{x_\lambda\}_{\lambda\in \Lambda }]\longrightarrow S
    \]
such that $\phi$ is surjective and $\phi|_R=\varphi$.
    \item Let $I=\ker\phi$, $I=(\{a_\mu\}_{\mu\in M})$. Adjoin variables $T_{1\mu}$ of degree $1$ with $\delta T_{1\mu}=a_\mu$ for $\mu\in M$. Then
    \[
    K=R[\{x_\lambda\}]\langle T_{1\mu}\mid \delta T_{1\mu}=a_\mu\rangle
    \]
    is a Koszul complex with $\hh_0(K)=S$.    
    \item 
    Adjoin variables of degree $2$ to $K$ in order to kill cycles in $\hh_1(K)$ and so on.
\end{enumerate}
The final result is a resolvent $X$ of $\varphi$. Of course the construction of $X$ is not at all unique. We are now very close to the definition of $\aq_i$.

Given a resolvent $X$ of $\varphi$, we define the module of differentials $(\Omega_{X/R},d)$ as an $X$-module together with a derivation $d\colon X\to \Omega_{X/R}$ such that the usual universal property holds: for any graded $X$-module $M$ and any $R$-linear derivation $\alpha\colon X\to M$, there is a unique $X$-module homomorphism $f\colon \Omega_{X/R}\to M$ such that $\alpha = fd$.

There is a unique $R$-linear homomorphism 
\[
\delta\colon \Omega_{X/R}\longrightarrow \Omega_{X/R}
\]
of degree $-1$ that is compatible with $d$, that is to say, $d\delta =\delta d$. Moreover this satisfies $\delta^2=0$, making $\Omega_{X/R}$ a complex. 

The pair $(\Omega_{X/R},d)$ can be constructed as follows: If $X=R\langle\{T_\lambda\}_{\lambda\in \Lambda}\rangle$, then $\Omega_{X/R}$ is a free $X$-module with basis $\{dT_\lambda\}_{\lambda\in \Lambda}$ and
\[
d\colon X\longrightarrow \Omega_{X/R}
\]
is uniquely determined by the requirement that it is $R$-linear and that 
\[
d(ab)= ad(b) + (-1)^{(\deg b)(\deg a) }b d(a)
\]
for $a,b\in X$ and homogeneous.

\begin{theorem}
    \label{4.1}
    Let $\varphi\colon R\to S$ be a ring homomorphism, $M$ an $S$-module and $X$ a resolvent of $\varphi$. Then the homology of $\Omega_{X/R}\otimes_XM$ is independent of the choice of $X$.\footnote{In fact, the complex $\Omega_{X/R}\otimes_XM$ is well-defined up to homotopy; see \citerecent{Quillen:1968c}.}
\end{theorem}

Here $M$ is viewed as a (differential graded) module over $X$ via the augmentation $\varepsilon \colon X\to S$. In particular, one has an isomorphism of $S$-complexes
\[
\Omega_{X/R}\otimes_X M \cong (\Omega_{X/R}\otimes_XS)\otimes_SM\,.
\]

For the proof of \eqref{4.1} it is essential that $\mathbb Q \subseteq R$. A correct proof of \eqref{4.1} is given by Bingener~\cite{Bingener:1987}. We define
\[
\aq_i(S/R,M) \coloneq \hh_i(\Omega_{X/R}\otimes_XM)\,.
\]

We now restrict our attention mainly to the following situation: $(R,\fm,k)$ is a noetherian local ring, $I\subset R$ a proper ideal, $S=R/I$ and $\varphi\colon R\to S$ the canonical epimorphism.

In that case the resolvent $X$ of $\varphi$ can be chosen such that $X_0=R$ and $X_i$ is a finitely generated free $R$-module for all $i\ge 0$. Note that $X$ is a free $R$-resolution of $S$. It is however almost never minimal in the sense that $\delta X\subseteq \fm X$.

Denote by $F_iX$ the subalgebra of $X$ generated by all elements of degree $\le i$. For each $i$ we have
\[
F_iX=F_{i-1}X\langle T_{i1},\dots,T_{ie_i}\rangle 
\]
where the $T_{ij}$ kill the homology $\hh_{i-1}(F_{i-1}X)$. We call $X$ a \emph{minimal resolvent} if 
\[
e_1=\mu_R(I) \quad\text{and}\quad e_i = \mu_S(\hh_{i-1}(F_{i-1}X)) \quad\text{for all $i\ge 2$.}
\]
We will see later that for a minimal resolvent the numbers $e_i$ coincide with the deviations of $S$, when $R$ is regular.

In the following the resolvent $X$ of $\varphi\colon R\to S$ need not be minimal. 

Let $\varepsilon\colon X\to S$ be the augmentation homomorphism and let $J = \ker(\varepsilon)$ be the augmentation ideal; it is an exact subcomplex of $X$, that is to say, $\hh(J)=0$. Furthermore we have that $J_0=I$ and $J_i=X_i$ for $i>0$.

Since $X$ is an algebra and $J$ is an ideal we can form the square of $J$. We thus obtain a subcomplex $J^2\subseteq J$ such that $J^2_0=I^2$ and $J^2_i=(F_{i-1}X)_i+ I(F_iX)_i$ for $i>0$. Since the quotient  $J/J^2$ is a complex of $S$-modules it follows that $\hh_i(J/J^2)$ is a finitely generated $S$-module for all $i$. We have
\begin{align*}
    (J/J^2)_0=I/I^2 \ \text{ and }\  (J/J^2)_i & = (F_i X/F_{i-1}X)_i\otimes_RS\\
    & = (X/F_{i-1}X)_i\otimes_RS
\end{align*}
for $i>0$. In particular $(J/J^2)_i$ are free $S$-modules for $i>0$.

\begin{lemma}
    There is an exact sequence of complexes of $S$-modules
    \[
    0 \longrightarrow I/I^2 \longrightarrow J/J^2 \longrightarrow \Omega_{X/R}\otimes_XS\longrightarrow 0.
    \]
    Here we consider $I/I^2$ to be a complex concentrated in degree zero.
\end{lemma}

\begin{proof}
    The morphism of complexes
    \[
    J \longrightarrow\Omega_{X/R}\otimes_XS,\quad x\mapsto dx\otimes 1
    \]
    induces the morphism of complexes
    \[
    \delta\colon J/J^2 \longrightarrow\Omega_{X/R}\otimes_XS.
    \]
    It is obvious that $\delta$ is surjective and that $\ker(\delta_0)=I/I^2$. Thus it remains to prove that
$\delta_i$ is injective for $i>0$. Suppose that $F_iX=F_{i-1}X\langle T_1,\ldots,T_e\rangle$. Then
\[
(J/J^2)_i = \bigoplus_{j=1}^{e} ST_j\quad \text{and}\quad(\Omega_{X/R}\otimes_XS)_i=\bigoplus_{j=1}^{e} SdT_j.
\]Moreover $\delta_i$ maps the basis element $T_j$ to $dT_j$ and is $S$-linear, hence the conclusion.
\end{proof}

The preceding lemma allows us to compute the $\aq_i$ in terms of the complex $J/J^2$.

\begin{corollary}\label{4.3}
In the context above one has
\[
\aq_0(S/R,S)=0, \quad \aq_1(S/R,S)=I/I^2, \ \text{ and } \aq_i(S/R,S)=\hh_i(J/J^2)
\]
for $i>0$. Let $L=\coker(I/I^2\to J/J^2)$. The complex $(L,\delta)$ is isomorphic to $\Omega_{X/R}\otimes_XS$. If we choose a minimal resolvent $X$, then $\delta(L)\subseteq \fm L$\footnote{This can be verified using Nakayama's lemma. In fact, $X$ is a minimal resolvent if and only if $\delta(L)\subseteq \fm L$.}. Hence if $\aq_i(S/R,S)=0$ for $i\geq 2$, then $L$\footnote{When shifted down by one homological degree.} is a minimal free $S$-resolution of $I/I^2$. \qed
\end{corollary}

In any case $L$ is a natural complex which resolves $I/I^2$ approximately. It is therefore of interest to compute its homology, which is $\aq(S/R,S)$.

In the sequel the following remarks will be helpful:
\begin{enumerate}[(a)]
    \item\label{item_resolvent_a}
$(F_iX/F_{i-1}X)_j=0$ for $j<i$,
    \item\label{item_resolvent_b} $\hh_j(F_iX)=0$ for $0<j<i$,
    \item\label{item_resolvent_c} If $F_iX=F_{i-1}X\langle T_1,\ldots,T_e\rangle$ then, assuming $i>1$, one has
    \[
    \hh_i(F_iX/F_{i-1}X)\cong (F_iX/F_{i-1}X)_i\otimes_R S\cong (J/J^2)_i\cong \bigoplus_{j=1}^{e} ST_j\,,
    \]
    \item\label{item_resolvent_d} $\hh_j(F_iX/F_{i-1}X)=0$ for $i<j<2i-1$,
    \item\label{item_resolvent_e} $\hh_{2i-1}(F_iX/F_{i-1}X)= \big(\bigoplus_{j=1}^{e} \hh_{i-1}(F_{i-1}X)T_j\big)/U$ for $i>0$,
where $U$ is generated by the following elements
\[
[x_j]T_k+(-1)^i[x_k]T_j,
\]
with $j,k=1,\ldots ,e$ and $\delta(T_j)=x_j$. 
\end{enumerate}

Now for all $i>0$ let
\[
\eta_i\colon \hh_i(F_iX)\longrightarrow (F_iX/F_{i-1}X)_i\otimes_RS
\]
be the natural homomorphism assigning to the class of a cycle $z = x+\sum a_jT_j$, $x\in F_{i-1}X$, the element $\sum\overline{a}_jT_j$, where $\overline{a}_j$ denotes the residue class of $a_j$ modulo $I$.

\begin{lemma}\label{lem_4.4}
    $\aq_{i}(S/R,S)\cong \ker(\eta_{i-1})$ for $i\geqslant 2$.
\end{lemma}

\begin{proof}
    Using \ref{item_resolvent_b}, \ref{item_resolvent_a}, and \ref{item_resolvent_c} we obtain for $i>1$ from the exact sequence
    \[
    0\longrightarrow F_{i-1}X\to F_iX \longrightarrow F_iX/F_{i-1}X\longrightarrow0
    \]
    the exact sequence
    \begin{equation}\label{eq_eta}
         \hh_i(F_iX)\xrightarrow{\ \eta_i\ } (F_iX/F_{i-1}X)_i\otimes_RS \xrightarrow{\ \sigma_i\ }\hh_{i-1}(F_{i-1}X) \longrightarrow 0.
    \end{equation}
    A simple computation shows that for $i>0$ the following diagram commutes
    \[
    \begin{tikzcd}
        (F_iX/F_{i-1}X)_i\otimes_RS \ar[d,"\cong"'] \ar[r,"\sigma_{i}"]&  \hh_{i-1}(F_{i-1}X) \ar[r,"\eta_{i-1}"]&  (F_{i-1}X/F_{i-2}X)_{i-1}\otimes_RS\ar[d,"\cong"]\\
        L_i \ar[rr,"\delta_i"] & & L_{i-1}.
    \end{tikzcd}
    \]
    Hence for all $i\geqslant 2$ we have 
    \begin{align*}
    T_{i}(S/R,S) & \cong \ker(\delta_{i})/\im(\delta_{i+1})\\
    & \cong \ker(\eta_{i-1}\sigma_i)/\im(\eta_{i}\sigma_{i+1})\\
    &\cong \sigma^{-1}_i\!\left[\ker(\eta_{i-1})\right]/\im(\eta_{i})\\
    &\cong \ker(\eta_{i-1}).
    \end{align*}
    For the last two isomorphisms we used \eqref{eq_eta}.
\end{proof}

For the long exact sequence
\[
\cdots \longrightarrow \hh_{i+1}(F_iX/F_{i-1}X)\longrightarrow\hh_i(F_{i-1}X)\longrightarrow \hh_i(F_iX) \xrightarrow{\ \eta_i\ }\cdots
\]
and from \ref{item_resolvent_b} and \cref{lem_4.4} we derive an alternative expression for $\aq_i$.

\begin{corollary}
\label{4.5}
    $\aq_{i+1}(S/R,S) \cong \hh_i(F_{i-1}X)$ for $i\geq 3$. \qed
\end{corollary}

Our next objective is to show that the low degree $\aq_i$-modules can be expressed in terms of Koszul homology and Tor. By construction $F_1X$ is the Koszul complex on a system of generators for $I$. We write  $\hh_i=\hh_i(F_1X)$.

\begin{lemma}\label{4.6}
    There is an exact sequence
    \[
\hh_3 \longrightarrow \aq_4(S/R,S)\longrightarrow {\textstyle\bigwedge^2}\hh_1 \longrightarrow \hh_2\longrightarrow \aq_3(S/R,S)\longrightarrow 0.
    \]
\end{lemma}

\begin{proof}
    Consider the long exact sequence
\[
\hh_3 \longrightarrow \hh_3(F_2X)\xrightarrow{\ \eta_3\ } \hh_3(F_2X/F_1X) \longrightarrow\hh_2\longrightarrow\hh_2(F_2X)\xrightarrow{\ \eta_2\ }
\]
derived from $0\to F_1X\to F_2X\to F_2X/F_1X\to 0$. The assertion follows from \ref{item_resolvent_e} and \cref{4.5}.  
\end{proof}

Since in the above sequence the image of ${\textstyle\bigwedge^2}\hh_1 \to \hh_2$ is exactly $\hh_1^2$ we get

\begin{corollary}\label{4.7}
    $\aq_3(S/R,S)\cong \hh_2/\hh_1^2$.\qed
\end{corollary}

Next we want to compare $\aq(S/R,S)$ with $\tor^R(S,S)$. The complex
\[
\cdots\longrightarrow J_3\longrightarrow  J_2\longrightarrow  J_1\longrightarrow 0
\]
is a free $R$-resolution of $I$. 
Therefore $\hh_0(J/IJ)=0$ and 
\[
\hh_i(J/IJ)\cong \tor^R_{i-1}(S,I)\cong \tor^R_i(S,S) \ \text{ for }\ i>0.
\]
$\tor^R(S,S)$ is a skew-symmetric algebra, whose algebra structure is induced by the structure of $X$. We let $\tor^R_+(S,S)=\bigoplus_{i>0}\tor^R_i(S,S)$. From the exact sequence of complexes
\[
0\longrightarrow J^2/IJ\longrightarrow J/IJ\longrightarrow J/J^2\longrightarrow 0
\]
and \cref{4.3} we obtain immediately the following result.

\begin{lemma}
    \begin{enumerate}[(a)]
        \item There exists a long exact sequence
        \[
        \begin{tikzcd}[row sep =2ex]
            \cdots \ar[r] & \hh_i(J^2/IJ) \ar[r] & \tor^R_i(S,S) \ar[r]
             \ar[draw=none]{d}[name=X, anchor=center]{} & \aq_i(S/R,S) \ar[rounded corners,
            to path={ -- ([xshift=3ex]\tikztostart.east)
                      |- (X.center) \tikztonodes
                      -| ([xshift=-2ex]\tikztotarget.west)
                      -- (\tikztotarget)}]{dll}\\
            & \hh_{i-1}(J^2/IJ) \ar[r] & \cdots \ar[r]
             \ar[draw=none]{d}[name=Y, anchor=center]{}& \aq_3(S/R,S) \ar[rounded corners,
            to path={ -- ([xshift=3ex]\tikztostart.east)
                      |- (Y.center) \tikztonodes
                      -| ([xshift=-3ex]\tikztotarget.west)
                      -- (\tikztotarget)}]{dll} \\
            & \hh_2(J^2/IJ) \ar[r] & \tor^R_2(S,S) \ar[r] & \aq_2(S/R,S) \ar[r] & 0.
        \end{tikzcd}
        \]
    \item $\tor_+^R(S,S)^2\subseteq \ker\!\left(\tor^R(S,S)\to \aq^R(S/R,S)\right)$.\qed
    \end{enumerate}
\end{lemma}

We can easily compute $ \hh_2(J^2/IJ)$. Let 
\[
F_1X=R\langle T_{11},\ldots,T_{1e_1}, \delta(T_{1i})=a_i\rangle \ \text{ and } \ F_2X=F_1X\langle T_{21},\ldots,T_{2e_2}, \delta(T_{2i})=z_i\rangle.
\]
We have an exact sequence
\[
V_1\oplus V_2 \xrightarrow{\ \phi\ } \bigoplus_{i<j} ST_{1i}T_{1j} \longrightarrow \hh_2(J^2/IJ)\longrightarrow 0
\]
where $V_1= \bigoplus_{i,k} ST_{1i}T_{2k}$ and $V_2=\bigoplus_{i<j<k} ST_{1i}T_{1j}T_{1k}$, 
with 
\[
\phi(V_2)=0\quad\text{and}\quad \phi(T_{1i}T_{2k})= T_{1i}z_k.
\]
From this it follows that
\[
\hh_2(J^2/IJ)\cong {\textstyle\bigwedge^2}I/I^2\cong {\textstyle\bigwedge^2}\tor^R_1(S,S).
\]

\begin{corollary}
    There is an exact sequence
        \[
        \begin{tikzcd}[row sep =2ex, column sep = 4ex]
             \cdots \ar[r] & {\displaystyle\frac{\tor_3^R(S,S)}{\tor_1^R(S,S)^3+\tor_1^R(S,S)\tor_2^R(S,S)}} \ar[r]
             \ar[draw=none]{d}[name=Y, anchor=center]{}& \aq_3(S/R,S) \ar[rounded corners,
            to path={ -- ([xshift=3ex]\tikztostart.east)
                      |- (Y.center) \tikztonodes
                      -| ([xshift=-3ex]\tikztotarget.west)
                      -- (\tikztotarget)}]{dll} \\
             {\textstyle\bigwedge^2}I/I^2 \ar[r] & \tor_2^R(S,S) \ar[r] & \aq_2(S/R,S) \ar[r] & 0. \qed
        \end{tikzcd}
        \]
\end{corollary}

Comparing this result with \cref{4.7} we get

\begin{corollary}\label{4.10}
    If $\aq_2(S/R,S)=0$ and $\hh_2=\hh_1^2$
    then
    \[
    \tor^R_2(S,S)= {\textstyle \bigwedge^2}I/I^2.\qed
    \]
\end{corollary}
The condition $\aq_2(S/R,S)=0$  means that $I$ is syzygetic. 

We now use the above  lemmas to show in some cases the vanishing or non-vanishing of the $\aq_i$.

\begin{theorem}\label{4.11}
    Suppose $R$ is a local Cohen-Macaulay ring and $I\subseteq R$ is a perfect ideal of height $2$, and set $S=R/I$.  Then
\begin{enumerate}[\rm (a)]
    \item $\aq_2(S/R,S) =0$\footnote{The vanishing of $\aq_2(S/R,S)$ only holds under the additional assumption that $I$ is generically complete intersection; see \cite[Theorem (3.2)]{Avramov/Herzog:1980}.} and $\aq_3(S/R,S)=0$;
    \item $\aq_4(S/R,S) =0$ if and only if $I$ is a complete intersection ideal.
\end{enumerate}    
\end{theorem}

\begin{proof}
    Suppose that $I$ is minimally generated by $n+1$ elements. In \cite[Theorem (2.1)]{Avramov/Herzog:1980} it is proved the following:
    \begin{enumerate}
        \item $\hh_i$ is a Cohen-Macaulay module for $i\geq 0$;
        \item $\hh_i=\hh_1^i$ for $i>0$;
        \item $\ker(\bigwedge^i \hh_1 \to \hh_i)\neq 0$ for $i= 2,3,\ldots,n-1$.
    \end{enumerate}
    Here $\hh_i$ denotes as before the Koszul homology of a sequence $a_1,\ldots ,a_{n+1}$ generating $I$. From these facts and \ref{4.6} and \ref{4.7} the assertion follows.
\end{proof}

\begin{corollary}
    With the assumptions of \ref{4.11}, and the further assumption that $R$ is Gorenstein, we have
    \[
{\textstyle \bigwedge^2}I/I^2 \xrightarrow{\ \cong\ }\tor_2^R(S,S) \cong (K_S)^* \,,
    \]
    where $K_S$ is the canonical module of $S$, and where $(-)^*=\Hom_S(-,S)$.
\end{corollary}

\begin{proof}
    The first isomorphism follows from \ref{4.10} and \ref{4.11}. To prove the second isomorphism consider a free $R$-resolution of $I$:
    \[
    0\longrightarrow G \xrightarrow{\ \varphi\ } F\longrightarrow I \longrightarrow0.
    \]
We have $\tor_2^R(S,S)\cong \ker (\varphi\otimes_R S)$ and $K_S\cong \coker(\varphi^*\otimes_R S)$. This implies the assertion.
\end{proof}

\setcounter{subsection}{13}

\begin{theorem}\label{4.14}
    Suppose $I\subseteq R$ has finite projective dimension and is an almost complete intersection, and set $S=R/I$. Then: 
    \begin{enumerate}[\quad \rm(a)]
        \item $\aq_3(S/R,S)=0$;
        \item $I$ is a complete intersection ideal if and only if
        \[
        \aq_4(S/R,S)=0\text{ and } \aq_5(S/R,S)=0\,.
        \]
    \end{enumerate}
\end{theorem}

\begin{proof}
    Let again $\hh$  denote the homology on a minimal number of generators of $I$. Since $I$ is an almost complete intersection we have $\hh_i=0$ for $i>1$. By \ref{4.7} we get $\aq_3(S/R,S)=0$ and \ref{4.6} shows that $\aq_4(S/R,S)\cong \wedge^2 \hh_1$. 
    
    Suppose now that $\aq_4(S/R,S)=0$. Then necessarily $\mu(\hh_1)\le 1$. If $\hh_1=0$, then $I$ is a complete intersection. It remains therefore to consider the case where $\mu(\hh_1)=1$. Say the cycle $z$ generates $\hh_1$ and $F_2X=F_1X\langle T\mid \delta T=z\rangle$. We get the exact sequence
    \[
    0\longrightarrow \hh_2(F_2X) \longrightarrow ST \longrightarrow \hh_1\longrightarrow 0\,.
    \]
    $\hh_1$ cannot be free, since otherwise by \cite[Prop.\ 1.4.9]{Gulliksen/Levin:1969} $I$ would be a complete intersection and $\hh_1=0$. For each $x\in \mathrm{Ann} \hh_1$ let $z_x = \tilde{x} T - a_x$, where $\tilde{x}$ is a representation of $x$ in $R$ and $a_x\in F_1X$ is chosen such that $\delta a_x=\tilde{x}z$. It is then clear that the mapping
    \[
    \mathrm{Ann}\hh_1 \longrightarrow \hh_2(F_2X) \quad{\text{where}} \quad x\mapsto z_x
    \]
    is an isomorphism.  
    
    For any $x\in \mathrm{Ann}\hh_1$ we have $\delta(\tilde{x}T^2 - a_xT)= -a_xz$. Thus $a_xz$ is a cycle in $F_1X$ and since $\hh_3=0$, we can find $b_x\in F_1X$ such that $\delta b_x=a_xz$.

    For all $x\in \mathrm{Ann}\hh_1$ we have therefore constructed a cycle
    \[
    w_x = xT^2 - a_xT -b_x \in F_2X\,.
    \]
    Consider the inclusion 
    \[
    \hh_2(F_2X)^2\subseteq \hh_4(F_2X)\,.
    \]
    $\hh_2(F_2X)^2$ is generated by $[z_xz_y]$, where $x,y\in \mathrm{Ann}\hh_1$. It follows that $[w_x]\in \hh_4(F_2X)$ but $[w_x]\not\in \hh_2(F_2X)^2$ if $x\not\in (\mathrm{Ann}\hh_1)^2$. Since $(\mathrm{Ann}\hh_1)^2\ne \mathrm{Ann}\hh_1$, we conclude that 
    \[
    \hh_4(F_2X)/\hh_2(F_2X)^2\ne 0\,.
    \]
    We have the exact sequence
    \[
    \hh_5(F_3X/F_2X)\xrightarrow{\ \alpha\ } \hh_4(F_2X)\longrightarrow \hh_4(F_3X)\longrightarrow 0\,,
    \]
    which allows us to compute $\aq_5(S/R,S)$: From \ref{item_resolvent_e} (after \ref{4.3}) it follows that $\im\alpha = \hh_2(F_2X)^2$. Using Corollary \ref{4.5} we now find
    \[
    \pushQED{\qedhere}
    \aq_5(S/R,S)\cong \hh_4(F_3X) \cong \hh_4(F_2X)/\im \alpha \cong  \hh_4(F_2X)/\hh_2(F_2X)^2\ne 0\,.\qed
    \]
\end{proof}

There is another instance where we can say something about the $\aq_i$.

\begin{theorem}
\label{4.15}
Let $R$ be regular and $I\subseteq R$ an ideal such that each $R/I$-module has rational Poincar\'e series. The following conditions are equivalent:
\begin{enumerate}[\quad\rm(a)]
\item 
    $S=R/I$ is a complete intersection;
\item 
    $\aq_i(S/R,S)=0$ for $i>1$;
\item 
    There exists an integer $i_0>1$ such that $\aq_i(S/R,S)=0$ for $i\ge i_0$.
\end{enumerate}
\end{theorem}

Recall that the numbers 
\[
\beta_i(M)=\rank_k \tor^R_i(k,M)
\]
are called the \emph{Betti numbers} of $M$ and that
\[
\poin_M(z) = \sum_{i\geqslant 0} \beta_i(M)z^i
\]
is called the \emph{Poincar\'e series} of $M$.

It has been proved in many cases that $\poin_M(z)$ is a rational function; for instance, for Golod rings. But just recently there was found an example of a local ring for which $\poin_k(z)$ is not rational; see \cite{Anick:1980} and also \cite{Lofwall/Roos:1980}.

\begin{proof}[Proof of the theorem]
If $I$ is generated by a regular sequence $a_1,\dots,a_n$, then the Koszul complex 
\[
R\langle T_1,\dots,T_n\mid \delta(T_i)=a_i\rangle
\]
is a resolvent of $R\to S$ and therefore $L= 0\to \oplus ST_i\to 0$. This proves (a)$\Rightarrow$(b).

The implication (b)$\Rightarrow$(c) is trivial.

Let us prove (c)$\Rightarrow$(a): We choose a minimal resolvent $X$ of $R\to S$, then $\delta(L)\subseteq \fm L$, and therefore 
\[
\rank_R L_i = \rank_k \hh_i(L\otimes_Sk) =\rank_k \aq_i(S/R,k)\,.
\]
We now compute these numbers in a different way: To the sequence of ring homomorphisms
\[
R\longrightarrow S\longrightarrow k
\]
belongs the Zariski-sequence 
\[
\cdots \longrightarrow \aq_i(S/R,k) \longrightarrow \aq_i(k/R,k)\longrightarrow \aq_i(k/S,k)\longrightarrow\cdots
\]
We have $\aq_i(k/R,k)=0$ for $i>0$, since $R$ is regular. Therefore
\[
\aq_i(S/R,k)\cong \aq_{i+1}(k/S,k) \quad \text{for $i\ge 1$.}
\]
To compute $\aq_i(k/S,k)$ we choose a minimal resolvent $Y$ for $S\to k$, which by a theorem of T.~H.~Gulliksen~\cite{Gulliksen:1971}\footnote{See also Schoeller \citerecent{Schoeller:1967}.} is a minimal free $S$-resolution of $k$. 

The numbers
\[
\varepsilon_i =\varepsilon_i(S) = \rank_k (F_{i}Y/F_{i-1}Y)_{i}\otimes_S k
\]
are called the \emph{deviations} of $S$\footnote{The original manuscript used a different convention, with the $i$th deviation being $\rank_k (F_{i+1}Y/F_{i}Y)_{i+1}\otimes_S k$. We have adopted what is by now the mostly standard convention, in which $\varepsilon_1(S)$ is the embedding dimension of $S$, $\varepsilon_2(S)$ is the minimal number of generators of $I$, et cetera.}. Thus we obtain the following formula:
\[
\rank_S L_i = \varepsilon_{i+1}
\]
for $i>0$. Assume now that there exists an integer $i_0>1$ such that
\[
\aq_i(S/R,S)=0 \quad \text{for $i\ge i_0$.}
\]
Let $M=\mathrm{coker}(L_{i_0}\to L_{i_0-1})$. Then 
\begin{equation}
\label{eq:resM}
\cdots \longrightarrow L_{i_0}\longrightarrow L_{i_0-1} \longrightarrow M\longrightarrow 0    
\end{equation}
is a minimal free $S$-resolution of $M$ with
\[
\poin_M(z) = \sum_{i\geq i_0} \varepsilon_i z^{i-i_0}\,.
\]
We may assume that $\varepsilon_i>0$ for $i\ge i_0$, since if some $\varepsilon_i=0$ for $i\ge i_0$, then all $\varepsilon_j=0$ for $j\ge i$ because \eqref{eq:resM} is a minimal free resolution of $M$. However the vanishing of $\varepsilon_j$ for $j\ge i$ implies that $I$ is a complete intersection by \cite{Gulliksen:1971}. 

Now let  $F(z)=\sum_{i\geqslant 1}\varepsilon_iz^{i}$. By assumption $\poin_M(z)$ is rational, hence also $F(z)$ is rational. We also have that $\poin_k(z)$ is rational. On the other hand we show that $F(z)$ and $\poin_k(z)$ cannot be rational at the same time!

Since $Y$ is a minimal resolvent of $S\to k$ it is easy to see that
\[
\poin_k(z) = \frac{(1+z)^{\varepsilon_1}(1+z^3)^{\varepsilon_3}\cdots}{(1-z^2)^{\varepsilon_2}(1-z^4)^{\varepsilon_4}\cdots}
\]
Let $\lambda$ denote the logarithmic derivation, that is to say, $\lambda G=G'/G$. We have
\begin{align*}
(\lambda \poin_k)(-z) 
    & = \sum_{i=0}^\infty \frac{(-1)^{i+1}iz^{i-1}\varepsilon_i}{1-z^{i}} \\
    &= \sum_{i=0}^\infty \Big((-1)^{i+1}i\varepsilon_i \sum_{j=0}^\infty z^{ij+i-1} \Big)\\
    &= \sum_{i=0}^\infty \Big( \sum_{j|i+1} (-1)^{j+1 }j\varepsilon_{j}\Big)z^i \,.
\end{align*}
Let $\sum_{i=1}^\infty \alpha_iz^i = \big[(\lambda \poin_k)(-z) - F'(-z) - \sum_{i=0}^\infty \varepsilon_1z^i\big]z$. Then
$\sum_{i=0}^\infty \alpha_iz^i$ is a rational series and
\[
\alpha_i = \sum_{\genfrac{}{}{0pt}{}{j|i}{j\ne 1,i}} (-1)^{j+1} j\varepsilon_{j}\,.
\]
Since the $\varepsilon_i> 0$ for $i\ge i_0$, we find that for $i>i_0^2$ and $i$ odd:
\[
\alpha_i =
\begin{cases}
0 &\text{if $i$ is prime} \\
>0 & \text{if $i$ is not prime}.
\end{cases}
\]
Such a strange rational function doesn't exist by the following result.
\end{proof}

\begin{theorem}[Mahler~\cite{Mahler:2019}]
\label{th:mahler}
Let $\sum_{i=0}\alpha_iz^i$ be rational with $\alpha_i\in\mathbb Q$. There exists an integer $r>0$ and a subset $\{r_1,\dots,r_g\}\subseteq \{0,\dots,r-1\}$ such that for $i\gg 0$ one has $\alpha_i=0$ if and only if $i\equiv r_j\mod r$ for some $j\in\{1,\dots,g\}$. \qed 
\end{theorem}

Thus if our series were rational there would exist too many primes. Because of the results \eqref{4.11}, \eqref{4.14} and \eqref{4.15} we are led to conjecture the following:
\begin{conjecture}[C3]
    \label{C3}
    Let $R$ be a regular local ring and $S=R/I$. The following conditions are equivalent:
    \begin{enumerate}[\quad\rm(a)]
        \item 
        $I$ is a complete intersection;
        \item 
    $\aq_i(S/R,S)=0$ for $i>1$;
    \item 
        There exists an integer $i_0>1$ such that $\aq_i(S/R,S)=0$ for $i\ge i_0$.
    \end{enumerate}
\end{conjecture}

\appendix

\section{Recent developments}
\label{se:recent}
In this section we (Briggs and Iyengar) record the status of the conjectures discussed in the preceding sections. The main  development is that Conjecture (C1) has been proved in \citerecent[Theorem A]{Briggs:2022} in full generality. In fact the following more general statement is established in \emph{op.~cit.}

\begin{theorem}{\citerecent[Theorem 3.1]{Briggs:2022}}
\label{th:briggs}
Let $R\to S=R/I$ be a surjective homomorphism of  noetherian local rings with $\pdim_RS$ finite.  If there exists an $S$-module $N$ and an $S$-linear map $\alpha\colon I/I^2\to N$ such that $\alpha\otimes_Sk$ is one-to-one, where $k$ is the residue field of $S$, then $I$ is generated by a regular sequence. 
\end{theorem}

The added generality in the preceding statement is useful for it opens up a path towards proving  Conjecture (C2) as well. Indeed, given the theorem above, (C2) follows from a conjecture of Eisenbud and Mazur~\citerecent{Eisenbud/Mazur:1997} concerning evolutions. An equivalent formulation of their conjecture,  due to Lenstra, is that if $S$ is a local affine $k$-algebra, as in (C2), then in the Jacobi-Zariski sequence
\[
I/I^2\xrightarrow{\ \alpha \ } S\otimes_R \Omega_{R/k}\xrightarrow{\ \beta\ } \Omega_{S/k}\longrightarrow 0
\]
the induced surjective map $\alpha\colon I/I^2\to \ker(\beta)$ is such that $\alpha\otimes_Sk$ is an isomorphism; see~\citerecent[Proposition~1]{Eisenbud/Mazur:1997}. If this condition holds, then when $\Omega_{S/k}$ has finite projective dimension over $S$, so does $\ker(\beta)$, and hence we can apply Theorem~\ref{th:briggs} with $N=\ker(\beta)$ to deduce that $I$ is generated by a regular sequence. See also \citerecent[Theorem 3.4]{Briggs:2022}.

\subsection*{Higher cotangent modules}
The module of K\"ahler differentials and the conormal modules are the first two modules in a family of \emph{cotangent modules} introduced by Avramov and Herzog~\citerecent{Avramov/Herzog:1994} when $R$ is a field of characteristic $0$ and $S$ is a positively graded $R$-algebra, and extended to general maps in \citerecent{Briggs/Iyengar:2023}.  Namely, given a homomorphism of commutative noetherian rings $\varphi \colon R\to S$, and for $L$ the cotangent complex of $\varphi$, set
\[
\cotam n{\varphi}\coloneqq \mathrm{Coker}(d_{n+1})\,.
\]
This $S$-module is well-defined up to projective summands; see \citerecent[\S7.3]{Briggs/Iyengar:2023}. One has $\cotam 0{\varphi}=\Omega_{S\mid R}$ and when $S=R/I$, with $\varphi$ the natural surjection $R\to S$, then $\cotam 1{\varphi}=I/I^2$. Thus the next result, from \citerecent[Theorem~B]{Briggs/Iyengar:2023}, generalizes Theorem~\ref{th:briggs}.

\begin{theorem}
\label{th:rigid}
Let $\varphi\colon R\to S$ be a map of commutative noetherian rings, essentially of finite type and locally of finite flat dimension. If for some integer  $n\ge 1$ the $S$-module $\cotam n{\varphi}$  has finite flat dimension, then $\varphi$ is locally complete intersection.
\end{theorem}

As explained in \citerecent{Briggs/Iyengar:2023}, as a direct corollary of the preceding theorem one gets:

\begin{corollary}
\label{co:rigid}
    Let $\varphi\colon R\to S$ be a map of commutative noetherian rings that is locally of finite flat dimension. If $\aq_n(S/R,-) =0$ for some $n\ge 1$, then $\varphi$ is locally complete intersection.
\end{corollary}

Hitherto this result was known if $R$ contains $\mathbb Q$ as a subring, and  also when $\aq_n(S/R,-)=0$ for \emph{all} $n\gg 0$. Both results are due to Avramov~\citerecent{Avramov:1999a}; see also  Halperin~\citerecent{Halperin:1987} and Avramov and Halperin~\citerecent{Avramov/Halperin:1987}. The hypothesis of the latter result is equivalent to the finiteness of the flat dimension of $L_\varphi$, and in this form it settled a longstanding conjecture of Quillen~\citerecent{Quillen:1968c}.

\subsection*{Herzog's conjecture}
Next we turn to Conjecture (C3) due to Herzog that postulates that when $R$ be a regular local ring and $S=R/I$, if $\aq_n(S/R,S)=0$ for $n\gg 0$ then $I$ is generated by a regular sequence. In contrast with \cref{co:rigid}, it is not true that $\aq_n(S/R,S)=0$ for \emph{some} $n\ge 1$ implies $\varphi$ is locally complete intersection; see, for instance, ~\cref{4.14}.  Herzog's conjecture has been settled in some cases by Ulrich~\citerecent{Ulrich:1987}, including when $I$ is in the linkage class of a complete intersection, but remains wide open in general.

It follows from  Herzog's ~\cref{4.15} that Conjecture (C3) holds when the Poincar\'e series of every finitely generated $S$-module is rational. In fact,  Herzog's result is more precise: With $E^S_k(t) =  \sum_{n\geqslant 0}\varepsilon_nt^n$, the generating series of the deviations of $S$ introduced above, Herzog's argument yields:

\begin{theorem}
With $R$ and $S$ as above,  if $P^S_k(t)$ and $E^S_k(t)$ are both rational, then $I$ is generated by a regular sequence.
\end{theorem}

As Herzog notes, Roos  has discovered rings $S$ with $P^S_k(t)$  not rational~\cite{Lofwall/Roos:1980}. Roos has also found rings $S$ for which $P^S_k(t)$ is rational but there exist finitely generated $S$-modules $M$ for which $P^S_M(t)$ is not rational; see \citerecent{Roos:2005}. Much less is known about the series $E^S_k(t)$. In particular, the following questions seem to be open:

\begin{question}
Does there exist a local noetherian ring $S$ that is not a complete intersection for which $E^S_k(t)$ is rational? Is there a ring $S$ for which $P^S_k(t)$ and $E^S_k(t)$ are \emph{both} irrational?
\end{question}

Part of our motivation in considering this question is that, if it turns out $E^S_k(t)$ can never be rational when $S$ is not a complete intersection, one would get a new proof of Herzog's theorem. But the main motivation is really to come to better grips with the series $E^S_k(t)$.

Finally we note, by standard arguments, Herzog's conjecture is equivalent to the following more general statement:

\begin{conjecture}
Let $\varphi\colon R\to S$ be a map of commutative noetherian rings that is locally of finite flat dimension. If $\aq_n(S/R,S) =0$ for $n\gg 0$, then $\varphi$ is locally complete intersection.
\end{conjecture}

This formulation makes it clear that a positive solution to Herzog's conjecture would be  a significant strengthening of Avramov's result, confirming a conjecture of Quillen, mentioned above.

\begin{ack}
We are thankful to Samuel Alvite, Luchezar Avramov, and Antonino Ficarra for their comments, suggestions, and corrections. This material is based on work supported by the National Science Foundation under Grant No.\ DMS-1928930, while  Briggs and Iyengar were in residence at the Mathematical Sciences Research Institute in Berkeley, California, during the Spring semester of 2024. Iyengar was also partly supported by NSF grant DMS-200985.
\end{ack}

\bibliographystyle{amsplain}
\bibliography{Herzog.bib}

\bibliographystylerecent{amsplain}
\bibliographyrecent{Herzog.bib}

\end{document}